\documentclass{amsart}

\usepackage{amsmath,amsfonts,amssymb,amsmath,latexsym,multirow}

\newtheorem{theorem}{Theorem}[section]
\newtheorem{lemma}[theorem]{Lemma}
\newtheorem{corollary}[theorem]{Corollary}
\newtheorem{proposition}[theorem]{Proposition}

\theoremstyle{definition}

\theoremstyle{remark}
\newtheorem{remark}[theorem]{Remark}

\numberwithin{equation}{section}

\begin{document}

\title[$m$-multiple power moments of Kloosterman
           sums]{$\begin{array}{c}
                   \text{Recursive formulas generating power moments of} \\
                   \text{multi-dimensional Kloosterman sums and} \\
                   \text{$m$-multiple power moments of Kloosterman sums}
                 \end{array}$
           }

\author{dae san kim}
\address{Department of Mathematics, Sogang University, Seoul 121-742, Korea}
\curraddr{Department of Mathematics, Sogang University, Seoul
121-742, Korea} \email{dskim@sogong.ac.kr}
\thanks{This work was supported by National Research Foundation of Korea Grant funded by
the Korean Government 2009-0072514.}

%    General info
\subjclass[2000]{}

\date{}

\dedicatory{ }


\begin{abstract}
In this paper, we construct two binary linear codes associated with
multi-dimensional and $m -$multiple power Kloosterman sums (for any
fixed $m$) over the finite field $\mathbb{F}_{q}$. Here $q$ is a
power of two. The former codes are dual to a subcode of the binary
hyper-Kloosterman code. Then we obtain two recursive formulas for
the power moments of multi-dimensional Kloosterman sums and for the
$m$-multiple power moments of Kloosterman sums in terms of the
frequencies of weights in the respective codes. This is done via
Pless power moment identity and yields, in the case of power moments
of multi-dimensional Kloosterman sums, much simpler recursive
formulas than those associated with finite special linear groups
obtained previously.\\

Index terms-recursive formula, multi-dimensional Kloosterman sum,
Kloosterman sum, Pless power moment identity, weight distribution.\\

\end{abstract}

\maketitle
%%%%%%%%%%%%%%%%%%%%%%%%%%%%%%%%%%%%%%%%%%%%%%%%%%%%%%%%%%%%%%%%%%%%%%%%
\section{Introduction and Notations}
%%%%%%%%%%%%%%%%%%%%%%%%%%%%%%%%%%%%%%%%%%%%%%%%%%%%%%%%%%%%%%%%%%%%%%%%

Let $\psi$ be a nontrivial additive character of the finite field
$\mathbb{F}_q$ with $q=p^r$ elements ($p$ a prime), and let $m$ be a
positive integer. Then the $m$-dimensional Kloosterman sum $K_m(\psi
; a)$(\cite{LN1}) is defined by

\begin{align*}K_m(\psi ; a)=\sum_{\alpha_1 , \cdots ~, \alpha_m \in
\mathbb{F}_q^*} \psi(\alpha_1 + \cdots + \alpha_m + a \alpha_1^{-1}
\cdots  \alpha_m^{-1})~(a \in \mathbb{F}_q^*).
\end{align*}

For this, we have the Deligne bound

%(1)%%%%%%%%%%%%%%%%%%%%%%%%%%%%%%%%%%%%%%%%%%%%%%%%%%%%%%%%%%%%%%%%
\begin{align}\label{a}|K_m(\psi ; a)|\leq (m+1)q^{\frac{m}{2}}.
\end{align}
%%%%%%%%%%%%%%%%%%%%%%%%%%%%%%%%%%%%%%%%%%%%%%%%%%%%%%%%%%%%%%%%%%%%

In particular, if $m=1$, then $K_1(\psi ; a)$ is simply denoted by
$K(\psi ; a)$, and is called the Kloosterman sum. The Kloosterman
sum was introduced in 1926 \cite{K1} to give an estimate for the
Fourier coefficients of modular forms. It  has also been studied to
solve various problems in coding theory and cryptography over finite
fields of characteristic two.

For each nonnegative integer $h$, by $MK_m(\psi)^h$ we will denote
the $h$-th moment of the $m$-dimensional Kloosterman sum $K_m(\psi ;
a)$. Namely, it is given by

\begin{align*}
MK_m(\psi)^h=\sum_{a \in \mathbb{F}_q^*}K_m(\psi;a)^h.
\end{align*}

If $\psi=\lambda$ is the canonical additive character of
$\mathbb{F}_q$, then $MK_m(\lambda)^h$ will be simply denoted by
$MK_m^h$. If futher $m=1$, for brevity $MK_1^h$ will be indicated by
$MK^h$. The power moments of Kloosterman sums can be used, for
example, to give an estimate for the Kloosterman sums.

Explicit computations on power moments of Kloosterman sums were
begun with the paper \cite{S1} of Sali$\acute{e}$ in 1931, where he
showed, for any odd prime $q$,

\begin{align*}
MK^h=q^2 M_{h-1}-(q-1)^{h-1}+2 (-1)^{h-1} ~(h\geq1).
\end{align*}
Here $M_0=0$, and, for $h \in \mathbb{Z}_{>0}$,
\begin{align*}
M_h=|\{(\alpha_1,\cdots,\alpha_h)\in
(\mathbb{F}_q^*)^h|\sum_{j=1}^{h}\alpha_i=1=\sum_{j=1}^h
\alpha_i^{-1}\}|.
\end{align*}

For $q=p$ odd prime, Sali$\acute{e}$ obtained $MK^1$, $MK^2$,
$MK^3$, $MK^4$ in that same paper by determining $M_1$, $M_2$,
$M_3$. On the other hand, $MK^5$ can be expressed in terms of the
$p$-th eigenvalue for a weight $3$ newform on $\Gamma_0(15)$ (cf.
\cite{L1}, \cite{CM1}). $MK^6$ can be expressed in terms of the
$p$-th eigenvalue for a weight $4$ newform on $\Gamma_0(6)$
(cf.\cite{KJ1}). Also, based on numerical evidence, in \cite{E1}
Evans was led to propose a conjecture which expresses $MK^7$ in
terms of Hecke eigenvalues for a weight $3$ newform on
$\Gamma_0(525)$ with quartic nebentypus of conductor 105.

From now on, let us assume that $q=2^r$. Carlitz \cite{C1} evaluated
$MK^h$ for $h \leq 4$. Recently, Moisio was able to find explicit
expressions of $MK^h$, for $h\leq10$ (cf. \cite{M1}). This was done,
via Pless power moment identity, by connecting moments of
Kloosterman sums and the frequencies of weights in the binary
Zetterberg code of length $q+1$, which were known by the work of
Schoof and Vlugt in \cite{SM1}.

Also, Moisio considered binary hyper-Kloosterman codes
  ${C}(r,m)$ and determined the weight distributions of ${C}(r,m)$ and
  ${C}^{\bot}(r,m)$, for $r=2$ and all $m\geq2$, and for all $r\geq2$ and
  $m=3$ (cf. \cite{M2}). In \cite{MR1}, these results were further extended to the case of
  $r=3, 4$ and all $m\geq2$.

In this paper, along the line of \cite{D2} we construct two binary
linear codes ${C}_{n-1}$ and ${D}_m$, respectively connected with
multi-dimensional and $m$-multiple power Kloosterman sums (for any
fixed $m$) over the finite field $\mathbb{F}_q$. Here $q$ is a power
of two. The code ${C}_{n-1}^{\bot}$ is a subcode of the
hyper-Kloosterman code ${C}(r,n)$, which is mentioned above. Then we
obtain two recursive formulas for the power moments of
multi-dimensional Kloosterman sums and the $m$-multiple power
moments of Kloosterman sums in terms of the frequencies of weights
in the respective codes. This is done via Pless power moment
identity and yields, in the case of power moments of
multi-dimensional Kloosterman sums, much simpler recursive formulas
than those obtained previously in \cite{D1}.

%Theorem 1%%%%%%%%%%%%%%%%%%%%%%%%%%%%%%%%%%%%%%%%%%%%%%%%%%%%%%%%%%%%%
\begin{theorem}\label{A}
$(1)$ Let $n=2^s$, $q=2^r$. For $r\geq3$, and $h=1,2,\cdots,$

%(2)%%%%%%%%%%%%%%%%%%%%%%%%%%%%%%%%%%%%%%%%%%%%%%%%%%%%%%%%%%%%%%%%%%%%
\begin{align}\label{b}
\begin{split}&MK_{n-1}^{h}=\sum_{l=0}^{h-1}(-1)^{h+l+1}{\binom{h}{l}}
(q-1)^{(n-1)(h-l)}MK_{n-1}^{l}\\
            &\qquad +q\sum_{j=0}^{min\{(q-1)^{n-1},~h\}}(-1)^{h+j}
            C_{n-1,j}\sum_{t=j}^{h}t!S(h,t)2^{h-t}{\binom{(q-1)^{n-1}-j}
            {(q-1)^{n-1}-t}}.
\end{split}
\end{align}
%%%%%%%%%%%%%%%%%%%%%%%%%%%%%%%%%%%%%%%%%%%%%%%%%%%%%%%%%%%%%%%%%%%%%%%%
Here $S(h,t)$ indicates the Stirling number of the second kind given
by

%(3)%%%%%%%%%%%%%%%%%%%%%%%%%%%%%%%%%%%%%%%%%%%%%%%%%%%%%%%%%%%%%%%%%%%%%
\begin{align}\label{c}
S(h,t)=\frac{1}{t!}\sum_{j=0}^{t}(-1)^{t-j}{\binom{t}{j}}j^h.
\end{align}
%%%%%%%%%%%%%%%%%%%%%%%%%%%%%%%%%%%%%%%%%%%%%%%%%%%%%%%%%%%%%%%%%%%%%%%%%
In addition, $\{C_{n-1,j}\}_{j=0}^{(q-1)^{n-1}}$ denotes  the weight
distribution of the binary linear code $C_{n-1}$, given by

%(4)%%%%%%%%%%%%%%%%%%%%%%%%%%%%%%%%%%%%%%%%%%%%%%%%%%%%%%%%%%%%%%%%%%%%%%
\begin{align}\label{d}
C_{n-1,j}=\sum{\prod_{\beta\in\mathbb{F}_{q}}}{\binom{\delta(n-1,q;\beta)}
{\nu_{\beta}}},
\end{align}
%%%%%%%%%%%%%%%%%%%%%%%%%%%%%%%%%%%%%%%%%%%%%%%%%%%%%%%%%%%%%%%%%%%%%%%%%%
where the sum runs over all the sets of integers
$\{\nu_{\beta}\}_{\beta\in\mathbb{F}_{q}}(0\leq\nu_{\beta}\leq
\delta(n-1,q;\beta))$ satisfying

%(5)%%%%%%%%%%%%%%%%%%%%%%%%%%%%%%%%%%%%%%%%%%%%%%%%%%%%%%%%%%%%%%%%%%%%%%%
\begin{equation}\label{e}
\sum_{\beta\in\mathbb{F}_{q}}{\nu_{\beta}}=j, ~and~
\sum_{\beta\in\mathbb{F}_{q}}{\nu_{\beta}}{\beta}=0, ~and\\
\end{equation}
%%%%%%%%%%%%%%%%%%%%%%%%%%%%%%%%%%%%%%%%%%%%%%%%%%%%%%%%%%%%%%%%%%%%%%%%%%%
\begin{align*}
\begin{split}\nonumber
\delta(n-1,q;
\beta)&=|\{(\alpha_{1},\cdots,\alpha_{n-1})\in(\mathbb{F}_{q}^{*})^{n-1}|\\
&\qquad\qquad\qquad\alpha_{1}+\cdots+\alpha_{n-1}+\alpha_{1}^{-1}\cdots
\alpha_{n-1}^{-1}=\beta\}|\\
&=
\begin{cases}
q^{-1}\{(q-1)^{n-1}+1\},& \text {if $\beta=0$,}\\
K_{n-2}(\lambda;\beta^{-1})+q^{-1}\{(q-1)^{n-1}+1\},& \text {if
$\beta\in\mathbb{F}_{q}^{*}$}.
\end{cases}
\end{split}
\end{align*}
Here we understand that
$K_{0}(\lambda;\beta^{-1})=\lambda(\beta^{-1})$.\\

(2) Let $q=2^{r}$. For $r\geq3$, and $m,h=1,2,\cdots,$
%(6)%%%%%%%%%%%%%%%%%%%%%%%%%%%%%%%%%%%%%%%%%%%%%%%%%%%%%%%%%%%%%%%%%%%%%%%%
\begin{align}\label{f}
\begin{split}&MK_{}^{mh}=\sum_{l=0}^{h-1}(-1)^{h+l+1}{\binom{h}{l}}
(q-1)^{m(h-l)}MK^{ml}\\
            &\qquad\quad
            +q\sum_{j=0}^{min\{(q-1)^{m},h\}}(-1)^{h+j}D_{m,j}\sum_{t=j}^{h}
            t!S(h,t)2^{h-t}{\binom{(q-1)^{m}-j}{(q-1)^{m}-t}}.
\end{split}
\end{align}
%%%%%%%%%%%%%%%%%%%%%%%%%%%%%%%%%%%%%%%%%%%%%%%%%%%%%%%%%%%%%%%%%%%%%%%%%%%%%
Here $\{D_{m,j}\}_{j=0}^{(q-1)^{m}}$ is the weight distribution of
the binary linear code $D_{m}$, given by

%(7)%%%%%%%%%%%%%%%%%%%%%%%%%%%%%%%%%%%%%%%%%%%%%%%%%%%%%%%%%%%%%%%%%%%%%%%%%
\begin{align}\label{f2}
D_{m,j}=\sum\prod_{\beta\in\mathbb{F}_{q}}{\binom{\sigma(m,q;\beta)}{\nu_{\beta}}},
\end{align}
%%%%%%%%%%%%%%%%%%%%%%%%%%%%%%%%%%%%%%%%%%%%%%%%%%%%%%%%%%%%%%%%%%%%%%%%%%%%%
where the sum runs over all the sets of integers
$\{\nu_{\beta}\}_{\beta\in\mathbb{F}_{q}}(0\leq\nu_{\beta}\leq\sigma(m,q;\beta))$
satisfying (\ref{e}), and

%(8)%%%%%%%%%%%%%%%%%%%%%%%%%%%%%%%%%%%%%%%%%%%%%%%%%%%%%%%%%%%%%%%%%%%%%%%%%%%%%
\begin{align}\label{g}
\begin{split}
\sigma(m,q;\beta)&=|\{(\alpha_{1},\cdots,\alpha_{m})\in(\mathbb{F}_{q}^{*})^{m}|\\
&\qquad\qquad\qquad\qquad\alpha_{1}^{}+\cdots+\alpha_{m}^{}+\alpha_{1}^{-1}+\cdots
+\alpha_{m}^{-1}=\beta\}|\\
&=\sum\lambda(\alpha_{1}^{}+\cdots+\alpha_{m}^{})+q^{-1}\{(q-1)^{m}+(-1)^{m+1}\},
\end{split}
\end{align}
%%%%%%%%%%%%%%%%%%%%%%%%%%%%%%%%%%%%%%%%%%%%%%%%%%%%%%%%%%%%%%%%%%%%%%%%%%%%%%%%%%
with the sum running over all
$\alpha_{1}^{},\cdots,\alpha_{m}^{}\in\mathbb{F}_{q}^{*}$,
satisfying $\alpha_{1}^{-1}+\cdots+\alpha_{m}^{-1}=\beta$.\\

\end{theorem}
%%%%%%%%%%%%%%%%%%%%%%%%%%%%%%%%%%%%%%%%%%%%%%%%%%%%%%%%%%%%%%%%%%%%%%%%%%%

(1) and (2) of the following are respectively $n=2$ and $n=4$ cases
of Theorem \ref{A} (1) (cf. (\ref{x}), (\ref{y})), and (3) and (4)
are equivalent and $n=2$ case of Theorem \ref{A} (2) ((cf.
(\ref{l1}), (\ref{q1})).

%Corollary 2%%%%%%%%%%%%%%%%%%%%%%%%%%%%%%%%%%%%%%%%%%%%%%%%%%%%%%%%%%%%%%%%
\begin{corollary}\label{B}
$(1)$ Let $q=2^r$. For $r\geq3$, and $h=1,2,\cdots,$

%(9)%%%%%%%%%%%%%%%%%%%%%%%%%%%%%%%%%%%%%%%%%%%%%%%%%%%%%%%%%%%%%%%%%%%%%%%%%
\begin{align}\label{h}
\begin{split}
&MK_{}^{h}=\sum_{l=0}^{h-1}(-1)^{h+l+1}{\binom{h}{l}}(q-1)^{h-l}MK^{l}\\
            &\qquad\quad
            +q\sum_{j=0}^{min\{(q-1)^{},h\}}(-1)^{h+j}C_{1,j}\sum_{t=j}^{h}
            t!S(h,t)2^{h-t}{\binom{{q-1}^{}-j}{{q-1}^{}-t}},
\end{split}
\end{align}
%%%%%%%%%%%%%%%%%%%%%%%%%%%%%%%%%%%%%%%%%%%%%%%%%%%%%%%%%%%%%%%%%%%%%%%%%%%%%%
where $\{{C}_{1,j}\}_{j=0}^{q-1 }$ is the weight distribution of the
binary linear code ${C}_{1}$, with

\begin{align*}
{C}_{1,j}=\sum{\binom{1}{\nu_{0}}}\prod_{tr(\beta^{-1})=0}{\binom{2}{\nu_{\beta}}}
~(j=0,\cdots,N_{1}).
\end{align*}
Here the sum is over all the sets of nonnegative integers
$\{\nu_{0}\}\bigcup\{\nu_{\beta}\}_{tr(\beta^{-1})=0}$ satisfying
$\nu_{0}+\sum_{tr(\beta^{-1})=0}\nu_{\beta}=j$ and
$\sum_{tr(\beta^{-1})=0}\nu_{\beta}\beta=0$.\\

$(2)$ Let $q=2^{r}$. For $r\geq3$, and $h=1,2,\cdots,$
%(10)%%%%%%%%%%%%%%%%%%%%%%%%%%%%%%%%%%%%%%%%%%%%%%%%%%%%%%%%%%%%%%%%%%%%%%%%%%%
\begin{align}\label{i}
\begin{split}
&MK_{3}^{h}=\sum_{l=0}^{h-1}(-1)^{h+l+1}{\binom{h}{l}}(q-1)^{3(h-l)}MK_{3}^{l}\\
            &\qquad\quad
            +q\sum_{j=0}^{min\{(q-1)^{3},h\}}(-1)^{h+j}C_{3,j}\sum_{t=j}^{h}
            t!S(h,t)2^{h-t}{\binom{(q-1)^{3}-j}{(q-1)^{3}-t}},
\end{split}
\end{align}
%%%%%%%%%%%%%%%%%%%%%%%%%%%%%%%%%%%%%%%%%%%%%%%%%%%%%%%%%%%%%%%%%%%%%%%%%%%%%%%%%
where $\{{C}_{3,j}\}_{j=0}^{(q-1)^{3} }$ is the weight distribution
of the binary linear code ${C}_{3}$, with

\begin{align*}
{C}_{3,j}=\sum{\binom{m_{0}}{\nu_{0}}}\prod_{\substack{|t|<2
\sqrt{q}\\t\equiv-1(4)}}
\prod_{K(\lambda;\beta^{-1})=t}{\binom{m_t}{\nu_{\beta}}}.
\end{align*}
Here the sum runs over all the sets of nonnegative integers
$\{\nu_{\beta}\}_{\beta\in\mathbb{F}_{q}}$ satisfying (\ref{e}),

\begin{align*}
m_{0}=q^{2}-3q+3,
\end{align*}
and

\begin{align*}
m_{t}=t^{2}+q^{2}-4q+3,
\end{align*}
for every integer $t$ satisfying $|t|<2\sqrt{q}$ and $t\equiv-1(4)$
.\\

$(3)$ Let $q=2^{r}$. For $r\geq3$, and $h=1,2,\cdots,$
%(11)%%%%%%%%%%%%%%%%%%%%%%%%%%%%%%%%%%%%%%%%%%%%%%%%%%%%%%%%%%%%%%%%%%%%%%%%%%%%
\begin{align}\label{k}
\begin{split}
&MK_{}^{2h}=\sum_{l=0}^{h-1}(-1)^{h+l+1}{\binom{h}{l}}(q-1)^{2(h-l)}MK_{}^{2l}\\
            &\qquad\quad
            +q\sum_{j=0}^{min\{(q-1)^{2},h\}}(-1)^{h+j}D_{2,j}\sum_{t=j}^{h}
            t!S(h,t)2^{h-t}{\binom{(q-1)^{2}-j}{(q-1)^{2}-t}},
\end{split}
\end{align}
%%%%%%%%%%%%%%%%%%%%%%%%%%%%%%%%%%%%%%%%%%%%%%%%%%%%%%%%%%%%%%%%%%%%%%%%%%%%%%%%%%
where $\{D_{2,j}\}_{j=0}^{(q-1)^{2}}$ is the weight distribution of
the binary linear code ${D}_{2}$, with

%(12)%%%%%%%%%%%%%%%%%%%%%%%%%%%%%%%%%%%%%%%%%%%%%%%%%%%%%%%%%%%%%%%%%%%%%%%%%%%%%
\begin{align}\label{l}
\begin{split}
&D_{2,j}^{}=\sum{\binom{2q-3}{\nu_{0}}}\prod_{\beta\in\mathbb{F}_{q}^{*}}
{\binom{K(\lambda;\beta^{-1})+q-3}{\nu_{\beta}}}\\
            &\qquad
            =\sum{\binom{2q-3}{\nu_{0}}}\prod_{\substack{|t|<2\sqrt{q}\\t\equiv-1(4)}}
            \prod_{K(\lambda;\beta^{-1})=t}{\binom{t+q-3}{\nu_{\beta}}},
            with
\end{split}
\end{align}
%%%%%%%%%%%%%%%%%%%%%%%%%%%%%%%%%%%%%%%%%%%%%%%%%%%%%%%%%%%%%%%%%%%%%%%%%%%%%%%%%%%
the sum running over all the sets of nonnegative integers
$\{\nu_{\beta}\}_{\beta\in\mathbb{F}_{q}}$ satisfying (\ref{e}).\\

$(4)$ Let $q=2^{r}$. For $r\geq3$, and $h=1,2,\cdots,$
%(13)%%%%%%%%%%%%%%%%%%%%%%%%%%%%%%%%%%%%%%%%%%%%%%%%%%%%%%%%%%%%%%%%%%%%%%%%%%%%%%%%%
\begin{align}\label{m}
\begin{split}
&MK_{2}^{h}=\sum_{l=0}^{h-1}(-1)^{h+l+1}{\binom{h}{l}}(q^{2}-3q+1)^{(h-l)}MK_{2}^{l}\\
            &\qquad\quad
            +q\sum_{j=0}^{min\{(q-1)^{2},h\}}(-1)^{h+j}D_{2,j}\sum_{t=j}^{h}t!S(h,t)
            2^{h-t}{\binom{(q-1)^{2}-j}{(q-1)^{2}-t}},
\end{split}
\end{align}
%%%%%%%%%%%%%%%%%%%%%%%%%%%%%%%%%%%%%%%%%%%%%%%%%%%%%%%%%%%%%%%%%%%%%%%%%%%%%%%%%%%%%%
where $D_{2,j}(0\leq j \leq (q-1)^{2})$'s are just as in (\ref{l}).\\

\end{corollary}
%%%%%%%%%%%%%%%%%%%%%%%%%%%%%%%%%%%%%%%%%%%%%%%%%%%%%%%%%%%%%%%%%%%%%%%%%%%%%%%%%%%%%%

The next two theorems will be of use later.\\

%Theorem 3%%%%%%%%%%%%%%%%%%%%%%%%%%%%%%%%%%%%%%%%%%%%%%%%%%%%%%%%%%%%%%%%%%%%%%%%%%%%
\begin{theorem}\label{C}$($\cite{LW1}$)$
Let $q=2^{r}$, with $r \geq 2$. Then the range $R$ of
$K(\lambda;a)$, as $a$ varies over $\mathbb{F}_{q}^{*}$, is given by

\begin{align*}
R=\{t\in \mathbb{Z} ~|~|t|<2\sqrt{q},~t \equiv -1(mod 4)\}.
\end{align*}

In addition, each value $t\in R$ is attained exactly $H(t^{2}-q)$
times, where $H(d)$ is the Kronecker class number of $d$.\\

\end{theorem}
%%%%%%%%%%%%%%%%%%%%%%%%%%%%%%%%%%%%%%%%%%%%%%%%%%%%%%%%%%%%%%%%%%%%%%%%%%%%%%%%%%%%%%%%

%Theorem 4%%%%%%%%%%%%%%%%%%%%%%%%%%%%%%%%%%%%%%%%%%%%%%%%%%%%%%%%%%%%%%%%%%%%%%%%%%%%%%
\begin{theorem}\label{D}$($\cite{C2}$)$
For the canonical additive character $\lambda$ of $\mathbb{F}_{q}$,
and $a\in\mathbb{F}_{q}^{*}$,

%(14)%%%%%%%%%%%%%%%%%%%%%%%%%%%%%%%%%%%%%%%%%%%%%%%%%%%%%%%%%%%%%%%%%%%%%%%%%%%%%%%%%%%
\begin{align}\label{n}
K_{2}(\lambda;a)=K(\lambda;a)^{2}-q.
\end{align}
%%%%%%%%%%%%%%%%%%%%%%%%%%%%%%%%%%%%%%%%%%%%%%%%%%%%%%%%%%%%%%%%%%%%%%%%%%%%%%%%%%%%%%%%
\end{theorem}
%%%%%%%%%%%%%%%%%%%%%%%%%%%%%%%%%%%%%%%%%%%%%%%%%%%%%%%%%%%%%%%%%%%%%%%%%%%%%%%%%%%%%%%%

Before we proceed further, we will fix the notations that will be
used throughout this paper:

\begin{align*}
\begin{split}
q&=2^{r}~(r\in \mathbb{Z}_{>0}),\\
\mathbb{F}_{q}&= the~finite ~field ~with ~q ~elements,\\
tr(x)&=x+x^{2}+\cdots+x^{{2}^{r-1}} ~the~ trace~function~ \mathbb{F}_{q}~\rightarrow ~\mathbb{F}_{2},\\
\lambda(x)&=(-1)^{tr(x)} ~the~ canonical~ additive~ character~ of~
\mathbb{F}_{q}^{}.
\end{split}
\end{align*}

Note that any nontrivial additive character $\psi$ of
$\mathbb{F}_{q}$ is given by $\psi(x)=\lambda(ax)$, for a unique
$a\in\mathbb{F}_{q}^{*}$.\\

%%%%%%%%%%%%%%%%%%%%%%%%%%%%%%%%%%%%%%%%%%%%%%%%%%%%%%%%%%%%%%%%%%%%%%%%
\section{Construction of codes associated with multi-dimensional Kloosterman sums}
%%%%%%%%%%%%%%%%%%%%%%%%%%%%%%%%%%%%%%%%%%%%%%%%%%%%%%%%%%%%%%%%%%%%%%%%

We will construct binary linear codes ${C}_{n-1}$ of length
$N_{1}=(q-1)^{n-1}$, connected with the $(n-1)$-dimensional
Kloosterman sums. Here $n=2^{s}$, with $s\in \mathbb{Z}_{>0}$.

Let

%(15)%%%%%%%%%%%%%%%%%%%%%%%%%%%%%%%%%%%%%%%%%%%%%%%%%%%%%%%%%%%%%%%%%%%%
\begin{align}\label{o}
v_{n-1}=(\cdots,\alpha_{1}+\cdots+\alpha_{n-1}+\alpha_{1}^{-1}\cdots
\alpha_{n-1}^{-1},\cdots),
\end{align}
%%%%%%%%%%%%%%%%%%%%%%%%%%%%%%%%%%%%%%%%%%%%%%%%%%%%%%%%%%%%%%%%%%%%%%%%%
where $\alpha_{1}^{},\alpha_{2}^{},\cdots,\alpha_{n-1}^{}$ run
respectively over all elements of $\mathbb{F}_{q}^{*}$. Here we do
not specify the ordering of the components of $v_{n-1}$, but we
assume that some ordering is fixed.\\

%Proposition 5%%%%%%%%%%%%%%%%%%%%%%%%%%%%%%%%%%%%%%%%%%%%%%%%%%%%%%%%%%%
\begin{proposition}\label{E}$($\cite{D1}, Proposition 11$)$
For each $\beta \in \mathbb{F}_{q}$, let

\begin{align*}
\delta(n-1,q;\beta)=|\{(\alpha_{1}^{},\cdots,\alpha_{n-1}^{})\in(\mathbb{F}_{q}^{*})^{n-1}|
\alpha_{1}^{}+\cdots,+\alpha_{n-1}^{}+\alpha_{1}^{-1}\cdots\alpha_{n-1}^{-1}=\beta\}|
\end{align*}
(Note that $\delta(n-1,q;\beta)$ is the number of components with
those equal to $\beta$ in the vector $v_{n-1}$(cf.(\ref{o}))).\\
Then
\begin{align*}
\delta(n-1,q;0)=q^{-1}\{(q-1)^{n-1}+1\},
\end{align*}
and, for $\beta\in\mathbb{F}_{q}^{*}$,

\begin{align*}
\delta(n-1,q;\beta)=K_{n-2}(\lambda;\beta^{-1})+q^{-1}\{(q-1)^{n-1}+1\},
\end{align*}
where $K_{0}(\lambda;\beta^{-1})=\lambda(\beta^{-1})$ by
convention.\\

\end{proposition}
%%%%%%%%%%%%%%%%%%%%%%%%%%%%%%%%%%%%%%%%%%%%%%%%%%%%%%%%%%%%%%%%%%%%%%%%%%

%Corollary 6%%%%%%%%%%%%%%%%%%%%%%%%%%%%%%%%%%%%%%%%%%%%%%%%%%%%%%%%%%%%%%
\begin{corollary}\label{F}
$(1)$
%(16)%%%%%%%%%%%%%%%%%%%%%%%%%%%%%%%%%%%%%%%%%%%%%%%%%%%%%%%%%%%%%%%%%%%%%%
\begin{equation}\label{p}
\delta(1,q;\beta)=
\begin{cases}
2, & \text{if ~$tr(\beta^{-1})=0$},\\
1,& \text {if ~$\beta = 0$,}\\
0,& \text{if ~$tr(\beta^{-1}) =1$.}
\end{cases}
\qquad \qquad \qquad \qquad \qquad \qquad \qquad~~~~
\end{equation}
%%%%%%%%%%%%%%%%%%%%%%%%%%%%%%%%%%%%%%%%%%%%%%%%%%%%%%%%%%%%%%%%%%%%%%%%%%%

$(2)$
%(17)%%%%%%%%%%%%%%%%%%%%%%%%%%%%%%%%%%%%%%%%%%%%%%%%%%%%%%%%%%%%%%%%%%%%%%
\begin{equation}\label{q}
\delta(3,q;\beta)=
\begin{cases}
q^{2}-3q+3, & \text{if ~$\beta=0$},\\
K(\lambda;\beta^{-1})^{2}+q^{2}-4q+3,& \text {if ~$\beta \in
\mathbb{F}_{q}^{*}~(cf.(\ref{n}))$.}
\end{cases}
\end{equation}
%%%%%%%%%%%%%%%%%%%%%%%%%%%%%%%%%%%%%%%%%%%%%%%%%%%%%%%%%%%%%%%%%%%%%%%%%%%%

\end{corollary}
%%%%%%%%%%%%%%%%%%%%%%%%%%%%%%%%%%%%%%%%%%%%%%%%%%%%%%%%%%%%%%%%%%%%%%%%%%%%
The binary linear code ${C}_{n-1}$ is defined as

%(18)%%%%%%%%%%%%%%%%%%%%%%%%%%%%%%%%%%%%%%%%%%%%%%%%%%%%%%%%%%%%%%%%%%%%%%%
\begin{align}\label{r}
{C}_{n-1}=\{u\in\mathbb{F}_{2}^{N_1}|u\cdot v_{n-1}=0\},
\end{align}
%%%%%%%%%%%%%%%%%%%%%%%%%%%%%%%%%%%%%%%%%%%%%%%%%%%%%%%%%%%%%%%%%%%%%%%%%%%%
where the dot denotes the usual inner product in
$\mathbb{F}_q^{N_1}$.\\

The following Delsarte's theorem is well-known.\\

%Theorem 7%%%%%%%%%%%%%%%%%%%%%%%%%%%%%%%%%%%%%%%%%%%%%%%%%%%%%%%%%%%%%%%%%%
\begin{theorem}\label{G}$($\cite{MS1}$)$

Let $B$ be a linear code over $\mathbb{F}_{q}^{}$. Then

\begin{align*}
(B|_{\mathbb{F}_{2}})^{\bot}=tr(B^{\bot}).\\
\end{align*}
\end{theorem}
%%%%%%%%%%%%%%%%%%%%%%%%%%%%%%%%%%%%%%%%%%%%%%%%%%%%%%%%%%%%%%%%%%%%%%%%%%%%

In view of this theorem, the dual ${C}_{n-1}^{\bot}$ of ${C}_{n-1}$
is given by

%(19)%%%%%%%%%%%%%%%%%%%%%%%%%%%%%%%%%%%%%%%%%%%%%%%%%%%%%%%%%%%%%%%%%%%%%%%
\begin{align}\label{s}
{C}_{n-1}^{\bot}=\{c(a)=(\cdots,tr(a(\alpha_{1}^{}+\cdots+
\alpha_{n-1}^{}+\alpha_{1}^{-1}\cdots\alpha_{n-1}^{-1})),\cdots)|a\in\mathbb{F}_{q}\}.
\end{align}
%%%%%%%%%%%%%%%%%%%%%%%%%%%%%%%%%%%%%%%%%%%%%%%%%%%%%%%%%%%%%%%%%%%%%%%%%%%%

\indent
%Lemma 8%%%%%%%%%%%%%%%%%%%%%%%%%%%%%%%%%%%%%%%%%%%%%%%%%%%%%%%%%%%%%%%%%%%%
\begin{lemma}\label{F}

$(q-1)^{n-1}>nq^{\frac{n-1}{2}}$, for all $n=2^{s}~(s\in
\mathbb{Z}_{>0})$, and $ q=2^{r}\geq8$.

\end{lemma}

\begin{proof}
This can be proved, for example, by induction on $s$.
\end{proof}
%%%%%%%%%%%%%%%%%%%%%%%%%%%%%%%%%%%%%%%%%%%%%%%%%%%%%%%%%%%%%%%%%%%%%%%%%%%%%

\indent
%Proposition 9%%%%%%%%%%%%%%%%%%%%%%%%%%%%%%%%%%%%%%%%%%%%%%%%%%%%%%%%%%%%%%%
\begin{proposition}\label{H}
For $q=2^{r}$, with $r\geq3$, the map
$\mathbb{F}_{q}\rightarrow{C}_{n-1}^{\bot}(a\mapsto c(a))$ is an
$\mathbb{F}_{2}$-linear isomorphism.
\end{proposition}

\begin{proof}
The map is clearly $\mathbb{F}_{2}$-linear and onto. Let $a$ be in
the kernel of the map. Then
$tr(a(\alpha_{1}^{}+\cdots+\alpha_{n-1}^{}+\alpha_{1}^{-1}\cdots\alpha_{n-1}^{-1}))=0$,
 for all $\alpha_{1}^{},\cdots,\alpha_{n-1}^{}\in\mathbb{F}_{q}^{*}$.
 Suppose that $a\neq0$. Then, on the one hand,

%(20)%%%%%%%%%%%%%%%%%%%%%%%%%%%%%%%%%%%%%%%%%%%%%%%%%%%%%%%%%%%%%%%%%%%%%%%%%%%
\begin{align}\label{t}
\sum_{\alpha_{1}^{},\cdots,\alpha_{n-1}^{}\in\mathbb{F}_{q}^{*}}(-1)^{tr(a(\alpha_{1}^{}
+\cdots+\alpha_{n-1}^{}+\alpha_{1}^{-1}\cdots\alpha_{n-1}^{-1}))}=(q-1)^{n-1}=N_{1}.
\end{align}
%%%%%%%%%%%%%%%%%%%%%%%%%%%%%%%%%%%%%%%%%%%%%%%%%%%%%%%%%%%%%%%%%%%%%%%%%%%%%%%%
On the other hand, (\ref{t}) is equal to $K_{n-1}(\lambda;a)$ (cf.
proof of Proposition 11 in \cite{D1}), and so from Deligne's
estimate in (\ref{a}) we get
\begin{align*}
(q-1)^{n-1}\leq nq^{\frac{n-1}{2}}.
\end{align*}
But this is impossible for $q\geq8$, in view of Lemma \ref{F}.
\end{proof}
%%%%%%%%%%%%%%%%%%%%%%%%%%%%%%%%%%%%%%%%%%%%%%%%%%%%%%%%%%%%%%%%%%%%%%%%%%%%%%

%%%%%%%%%%%%%%%%%%%%%%%%%%%%%%%%%%%%%%%%%%%%%%%%%%%%%%%%%%%%%%%%%%%%%%%%
\section{Recursive formulas for power moments of multi-dimensional Kloostermann sums}
%%%%%%%%%%%%%%%%%%%%%%%%%%%%%%%%%%%%%%%%%%%%%%%%%%%%%%%%%%%%%%%%%%%%%%%%
\indent We are now ready to derive, via Pless power moment identity,
a recursive formula for the power moments of multi-dimensional
Kloosterman sums in terms of the frequencies of weights in
${C}_{n-1}$. \indent
%Theorem 10%%%%%%%%%%%%%%%%%%%%%%%%%%%%%%%%%%%%%%%%%%%%%%%%%%%%%%%%%%%%%%
\begin{theorem}\label{I}$($Pless power moment identity,
\cite{MS1}$)$ Let $B$ be an $q$-ary [$n,k$] code, and let
$B_{i}$(resp. $B_{i}^{\bot}$) denote the number of codewords of
weight $i$ in $B$(resp. in $B^{\bot}$). Then, for $h=0,1,2,\cdots,$

%(21)%%%%%%%%%%%%%%%%%%%%%%%%%%%%%%%%%%%%%%%%%%%%%%%%%%%%%%%%%%%%%%%%%%%%%
\begin{align}\label{u}
\sum_{i=0}^{n}~i^{h}B_{i}=\sum_{i=0}^{min\{n,h\}}(-1)^{i}B_{i}^{\bot}
\sum_{t=i}^{h}t!S(h,t)q^{k-t}(q-1)^{t-i}{\binom{n-i}{n-t}},
\end{align}
%%%%%%%%%%%%%%%%%%%%%%%%%%%%%%%%%%%%%%%%%%%%%%%%%%%%%%%%%%%%%%%%%%%%%%%%%%

where $S(h,t)$ is the Stirling number of the second kind defined
in(\ref{c}).\\

\end{theorem}
%%%%%%%%%%%%%%%%%%%%%%%%%%%%%%%%%%%%%%%%%%%%%%%%%%%%%%%%%%%%%%%%%%%%%%%%%%

For the following lemma, observe that $(n,q-1)=1$.\\

%Lemma 11%%%%%%%%%%%%%%%%%%%%%%%%%%%%%%%%%%%%%%%%%%%%%%%%%%%%%%%%%%%%%%%%%
\begin{lemma}\label{J}
The map $a \mapsto a^{n} : \mathbb{F}_{q}^{*}
\rightarrow\mathbb{F}_{q}^{*}$ is a bijection.
\end{lemma}
%%%%%%%%%%%%%%%%%%%%%%%%%%%%%%%%%%%%%%%%%%%%%%%%%%%%%%%%%%%%%%%%%%%%%%%%%%
\indent

%Lemma 12%%%%%%%%%%%%%%%%%%%%%%%%%%%%%%%%%%%%%%%%%%%%%%%%%%%%%%%%%%%%%%%%%
\begin{lemma}\label{K}
For $a\in\mathbb{F}_{q}^{*}$, the Hamming weight $w(c(a))$(cf. (\ref
{s})) of $c(a)$ can be expressed as follows:

%(22)%%%%%%%%%%%%%%%%%%%%%%%%%%%%%%%%%%%%%%%%%%%%%%%%%%%%%%%%%%%%%%%%%%%%%
\begin{align}\label{v}
w(c(a))={\frac{N_{1}}{2}}-{\frac{1}{2}}K_{n-1}(\lambda;a),~with~
N_{1}=(q-1)^{n-1}.
\end{align}
%%%%%%%%%%%%%%%%%%%%%%%%%%%%%%%%%%%%%%%%%%%%%%%%%%%%%%%%%%%%%%%%%%%%%%%%%%

\end{lemma}

\begin{proof}
\begin{equation*}
\begin{split}
w(c(a))&={\frac{1}{2}}\sum_{\alpha_{1}^{},\cdots,\alpha_{n-1}^{}
\in\mathbb{F}_{q}^{*}}(1-(-1)^{tr(a(\alpha_{1}^{}+\cdots+
\alpha_{n-1}^{}+\alpha_{1}^{-1}\cdots\alpha_{n-1}^{-1}))})\\
&={\frac{1}{2}}\{N_{1}-\sum_{\alpha_{1}^{},\cdots,\alpha_{n-1}^{}
\in\mathbb{F}_{q}^{*}}\lambda(a(\alpha_{1}^{}+\cdots+\alpha_{n-1}^{}
+\alpha_{1}^{-1}\cdots\alpha_{n-1}^{-1}))\}\\
&={\frac{N_{1}}{2}}-{\frac{1}{2}}\sum_{\alpha_{1}^{},\cdots,
\alpha_{n-1}^{}\in\mathbb{F}_{q}^{*}}\lambda(\alpha_{1}^{}
+\cdots+\alpha_{n-1}^{}+a^{n}\alpha_{1}^{-1}\cdots\alpha_{n-1}^{-1})\\
&={\frac{N_{1}}{2}}-{\frac{1}{2}}\sum_{\alpha_{1}^{},
\cdots,\alpha_{n-1}^{}\in\mathbb{F}_{q}^{*}}\lambda(\alpha_{1}^{n}
+\cdots+\alpha_{n-1}^{n}+a^{n}\alpha_{1}^{-n}\cdots\alpha_{n-1}^{-n})\\
&\qquad\qquad(by~ Lemma~ \ref{J})\\
&={\frac{N_{1}}{2}}-{\frac{1}{2}}\sum_{\alpha_{1}^{},
\cdots,\alpha_{n-1}^{}\in\mathbb{F}_{q}^{*}}\lambda((\alpha_{1}^{}
+\cdots+\alpha_{n-1}^{}+a^{}\alpha_{1}^{-1}\cdots\alpha_{n-1}^{-1})^{n})\\
&={\frac{N_{1}}{2}}-{\frac{1}{2}}\sum_{\alpha_{1}^{},
\cdots,\alpha_{n-1}^{}\in\mathbb{F}_{q}^{*}}\lambda(\alpha_{1}^{}
+\cdots+\alpha_{n-1}^{}+a^{}\alpha_{1}^{-1}\cdots\alpha_{n-1}^{-1})\\
\end{split}
\end{equation*}
~\qquad \qquad \qquad \qquad \qquad({\cite{LN1}}, Theorem 2.23(v))
\begin{align*}
={\frac{N_{1}}{2}}-{\frac{1}{2}}K_{n-1}(\lambda;a).\qquad \qquad
\qquad \qquad \qquad \qquad \qquad
\end{align*}
\end{proof}
%%%%%%%%%%%%%%%%%%%%%%%%%%%%%%%%%%%%%%%%%%%%%%%%%%%%%%%%%%%%%%%%%%%%%%%%%%

Denote for the moment $v_{n-1}$ in (\ref{o}) by
$v_{n-1}=(g_{1},g_{2},\cdots,g_{N_{1}})$. Let
$u=(u_{1},\cdots,u_{N_{1}})\in\mathbb{F}_{2}^{N_{1}}$, with
$\nu_{\beta}$ $1$'s in the coordinate places where $g_{l}=\beta$,
for each $\beta\in\mathbb{F}_{q}^{}$. Then we see from the
definition of the code ${C}_{n-1}$ (cf. (\ref{r})) that $u$ is a
codeword with weight $j$ if and only if
$\sum_{\beta\in\mathbb{F}_{q}}\nu_{\beta}=j$ and
$\sum_{\beta\in\mathbb{F}_{q}}\nu_{\beta}\beta=0$ (an identity in
$\mathbb{F}_{q}$). As there are
$\prod_{\beta\in\mathbb{F}_{q}}{\binom{\delta(n-1,q;\beta)}{\nu_{\beta}}}$
(cf. Proposition \ref{E}) many such codewords with weight $j$, we
obtain the following result.

%Proposition 13%%%%%%%%%%%%%%%%%%%%%%%%%%%%%%%%%%%%%%%%%%%%%%%%%%%%%%%%%%%%
\begin{proposition}\label{L}
Let $\{{C}_{n-1,j}\}_{j=0}^{N_{1}}$ be the weight distribution of
${C}_{n-1}$, where ${C}_{n-1,j}$ denotes the frequency of the
codewords with weight $j$ in ${C}_{n-1}$. Then

\begin{align*}
{C}_{n-1,j}=\sum\prod_{\beta\in\mathbb{F}_q}{\binom{\delta(n-1,q;\beta)}
{\nu_{\beta}}},
\end{align*}
where the sum runs over all the sets of integers
$\{\nu_{\beta}\}_{\beta\in\mathbb{F}_{q}}~(0\leq\nu_{\beta}\leq\delta(n-1,q;\beta))$
satisfying
\begin{align*}
\sum_{\beta\in\mathbb{F}_{q}}\nu_{\beta}=j,~
and~\sum_{\beta\in\mathbb{F}_{q}}\nu_{\beta}\beta=0.
\end{align*}

\end{proposition}
%%%%%%%%%%%%%%%%%%%%%%%%%%%%%%%%%%%%%%%%%%%%%%%%%%%%%%%%%%%%%%%%%%%%%%%%%%
\indent

%Corollary 14%%%%%%%%%%%%%%%%%%%%%%%%%%%%%%%%%%%%%%%%%%%%%%%%%%%%%%%%%%%%%%%%%
\begin{corollary}\label{M}

$(1)$ Let $\{{C}_{1,j}\}_{j=0}^{q-1}$ be the weight distribution of
${C}_{1}$. Then

%(23)%%%%%%%%%%%%%%%%%%%%%%%%%%%%%%%%%%%%%%%%%%%%%%%%%%%%%%%%%%%%%%%%%%%%%
\begin{align}\label{x}
{C}_{1,j}=\sum{\binom{1}{\nu_{0}}}\prod_{tr(\beta^{-1})=0}{\binom{2}{\nu_{\beta}}}~(j=0,\cdots,q-1),
\end{align}
%%%%%%%%%%%%%%%%%%%%%%%%%%%%%%%%%%%%%%%%%%%%%%%%%%%%%%%%%%%%%%%%%%%%%%%%%%
where the sum is over all the sets of nonnegative integers
$\{\nu_{0}\}\cup\{\nu_{\beta}\}_{tr(\beta^{-1})=0}$ satisfying
$\nu_{0}+\sum_{tr(\beta^{-1})=0}\nu_{\beta}=j$ and
$\sum_{tr(\beta^{-1})=0}\nu_{\beta}\beta=0$ (cf.(\ref{p})).\\

$(2)$ Let $\{{C}_{3,j}\}_{j=0}^{(q-1)^{3}}$ be the weight
distribution of ${C}_{3}$. Then

%%(24)%%%%%%%%%%%%%%%%%%%%%%%%%%%%%%%%%%%%%%%%%%%%%%%%%%%%%%%%%%%%%%%%%%%%
\begin{align}\label{y}
{C}_{3,j}=\sum{\binom{m_{0}}{\nu_{0}}}\prod_{\substack{|t|<2\sqrt{q}\\
t\equiv-1(4)}}\prod_{K(\lambda;\beta^{-1})=t}{\binom{m_{t}}{\nu_{\beta}}},
\end{align}
%%%%%%%%%%%%%%%%%%%%%%%%%%%%%%%%%%%%%%%%%%%%%%%%%%%%%%%%%%%%%%%%%%%%%%%%%%
where the sum runs over all the sets of integers
$\{\nu_{\beta}\}_{\beta\in\mathbb{F}_{q}}$ satisfying

\begin{align*}
\sum_{\beta\in\mathbb{F}_{q}}\nu_{\beta}=j,~and~\sum_{\beta\in\mathbb{F}_{q}}\nu_{\beta}\beta=0,
\end{align*}

\begin{align*}
m_{0}=q^{2}-3q+3,
\end{align*}
and
\begin{align*}
m_{t}=t^{2}+q^{2}-4q+3,
\end{align*}
for every integer $t$ satisfying $|t|<2\sqrt{q}$ and $t\equiv-1(4)$
(cf. Theorem \ref{C}, (\ref{q})).\\

\end{corollary}
%%%%%%%%%%%%%%%%%%%%%%%%%%%%%%%%%%%%%%%%%%%%%%%%%%%%%%%%%%%%%%%%%%%%%%%%%%

%Remark 15%%%%%%%%%%%%%%%%%%%%%%%%%%%%%%%%%%%%%%%%%%%%%%%%%%%%%%%%%%%%%%%%
\begin{remark}\label{N}
This shows that the weight distribution of ${C}_{1}$ is the
same as that of ${C}(SO^{+}(2,q))$ (cf. \cite{D3}).\\
\end{remark}
%%%%%%%%%%%%%%%%%%%%%%%%%%%%%%%%%%%%%%%%%%%%%%%%%%%%%%%%%%%%%%%%%%%%%%%%%%

From now on, we will assume that $r\geq3$, and hence every codeword
in ${C}_{n-1}^{\bot}$ can be written as $c(a)$, for a unique
$a\in\mathbb{F}_{q}$ (cf. Proposition \ref{H}).\\

We now apply the Pless power moment identity in (\ref{u}) to
${C}_{n-1}^{\bot}$,  in order to obtain the result in Theorem
\ref{A} (1) about a recursive formula. Then the left hand side of
that identity in (\ref{u}) is equal to

%(25)%%%%%%%%%%%%%%%%%%%%%%%%%%%%%%%%%%%%%%%%%%%%%%%%%%%%%%%%%%%%%%%%%%%%%
\begin{align}\label{z}
\sum_{a\in\mathbb{F}_{q}^{*}}w(c(a))^{h},
\end{align}
%%%%%%%%%%%%%%%%%%%%%%%%%%%%%%%%%%%%%%%%%%%%%%%%%%%%%%%%%%%%%%%%%%%%%%%%%%
with $w(c(a))$ given by (\ref{v}). So (\ref{z}) is

%(26)%%%%%%%%%%%%%%%%%%%%%%%%%%%%%%%%%%%%%%%%%%%%%%%%%%%%%%%%%%%%%%%%%%%%%
\begin{align}\label{a1}
\begin{split}
\sum_{a\in\mathbb{F}_{q}^{*}}w(c(a))^{h}&={\frac{1}{2^{h}}}\sum_{a\in
\mathbb{F}_{q}^{*}}(N_{1}-K_{n-1}(\lambda;a))^{h}\\
&={\frac{1}{2^{h}}}\sum_{a\in\mathbb{F}_{q}^{*}}\sum_{l=0}^{h}(-1)^{l}
{\binom{h}{l}}N_{1}^{h-1}K_{n-1}(\lambda;a)^{l}\\
&={\frac{1}{2^{h}}}\sum_{l=0}^{h}(-1)^{l}{\binom{h}{l}}N_{1}^{h-1}MK_{n-1}^{l}.
\end{split}
\end{align}
%%%%%%%%%%%%%%%%%%%%%%%%%%%%%%%%%%%%%%%%%%%%%%%%%%%%%%%%%%%%%%%%%%%%%%%%%%

On the other hand, noting that $dim_{\mathbb{F}_{2}}{C}_{n-1}=r$
(cf. Proposition \ref{H}) the right hand side of the Pless moment
identity(cf. (\ref{u})) becomes

%(27)%%%%%%%%%%%%%%%%%%%%%%%%%%%%%%%%%%%%%%%%%%%%%%%%%%%%%%%%%%%%%%%%%%%%%
\begin{align}\label{b1}
q\sum_{j=0}^{min\{N_{1},h\}}(-1)^{j}{C}_{n-1,j}\sum_{t=j}^{h}
t!S(h,t)2^{-t}{\binom{N_{1}-j}{N_{1}-t}}.
\end{align}
%%%%%%%%%%%%%%%%%%%%%%%%%%%%%%%%%%%%%%%%%%%%%%%%%%%%%%%%%%%%%%%%%%%%%%%%%%

Our result in (\ref{b}) follows now by equating (\ref{a1}) and
(\ref{b1}).\\

%Remark 16%%%%%%%%%%%%%%%%%%%%%%%%%%%%%%%%%%%%%%%%%%%%%%%%%%%%%%%%%%%%%%%%%
\begin{remark}\label{O}
A recursive formula for the power moments of multi-dimensional
Kloosterman sums was obtained in \cite{D1} by constructing binary
linear codes  ${C}(SL(n,q))$ and utilizing explicit expressions of
Gauss sums for the finite special linear group $SL(n,q)$. However,
our result in (\ref{b}) is better than that in (1) of \cite{D1}.
Because our formula here is much simpler than the one there. Indeed,
the length of the code ${C}_{n-1}$ here is $N_{1}=(q-1)^{n-1}$,
whereas that of ${C}(SL(n,q))$ there is
$N=q^{\binom{n}{2}}\prod_{j=2}^{n}(q^{j}-1)$, both of which
appear in their respective expressions of  recursive formulas.\\
\end{remark}
%%%%%%%%%%%%%%%%%%%%%%%%%%%%%%%%%%%%%%%%%%%%%%%%%%%%%%%%%%%%%%%%%%%%%%%%%%

%%%%%%%%%%%%%%%%%%%%%%%%%%%%%%%%%%%%%%%%%%%%%%%%%%%%%%%%%%%%%%%%%%%%%%%%%%
%%%%%%%%%%%%%%%%%%%%%%%%%%%%%%%%%%%%%%%%%%%%%%%%%%%%%%%%%%%%%%%%%%%%%%%%%%
\section{Construction of codes associated with powers of\\ Kloosterman
sums}
%%%%%%%%%%%%%%%%%%%%%%%%%%%%%%%%%%%%%%%%%%%%%%%%%%%%%%%%%%%%%%%%%%%%%%%%%%
%%%%%%%%%%%%%%%%%%%%%%%%%%%%%%%%%%%%%%%%%%%%%%%%%%%%%%%%%%%%%%%%%%%%%%%%%%

We will construct binary linear codes $D_{m}$ of length
$N_{2}=(q-1)^{m}$, connected with the $m$-th powers of (the
ordinary) Kloosterman sums. Here $m\in \mathbb{Z}_{>0}$.

Let
%(28)%%%%%%%%%%%%%%%%%%%%%%%%%%%%%%%%%%%%%%%%%%%%%%%%%%%%%%%%%%%%%%%%%%%%%
\begin{align}\label{c1}
w_{m}=(\cdots,\alpha_{1}^{}+\cdots+\alpha_{m}^{}+\alpha_{1}^{-1}+\cdots
+\alpha_{m}^{-1},\cdots),
\end{align}
%%%%%%%%%%%%%%%%%%%%%%%%%%%%%%%%%%%%%%%%%%%%%%%%%%%%%%%%%%%%%%%%%%%%%%%%%%
where $\alpha_{1}^{},\alpha_{2}^{},\cdots,\alpha_{m}^{}$ run
respectively over all elements of $\mathbb{F}_{q}^{*}$. Here we do
not specify the ordering of the components of $w_{m}$, but we assume
that some ordering is fixed.

%Theorem 17%%%%%%%%%%%%%%%%%%%%%%%%%%%%%%%%%%%%%%%%%%%%%%%%%%%%%%%%%%%%%%%%%%
\begin{theorem}\label{P}$($\cite{D3}$)$
Let $\lambda$ be the canonical additive character of
$~\mathbb{F}_q$, and let $\beta\in\mathbb{F}_{q}^{*}$. Then

%(29)%%%%%%%%%%%%%%%%%%%%%%%%%%%%%%%%%%%%%%%%%%%%%%%%%%%%%%%%%%%%%%%%%%%%%
\begin{align}\label{d1}
\sum_{\alpha\in\mathbb{F}_{q}-\{0,1\}}\lambda({\frac{\beta}{\alpha^{2}+\alpha}})=K(\lambda;\beta)-1.
\end{align}\\
%%%%%%%%%%%%%%%%%%%%%%%%%%%%%%%%%%%%%%%%%%%%%%%%%%%%%%%%%%%%%%%%%%%%%%%%%%

\end{theorem}
%%%%%%%%%%%%%%%%%%%%%%%%%%%%%%%%%%%%%%%%%%%%%%%%%%%%%%%%%%%%%%%%%%%%%%%%%%

%Proposition 18%%%%%%%%%%%%%%%%%%%%%%%%%%%%%%%%%%%%%%%%%%%%%%%%%%%%%%%%%%%%%%%%
\begin{proposition}\label{R}

For each $\beta\in\mathbb{F}_{q}$, let
\begin{align*}
\sigma(m,q;\beta)=|\{(\alpha_{1}^{},\cdots,\alpha_{m}^{})\in(\mathbb{F}_{q}^{*})^{m}|\alpha_{1}^{}+\cdots+\alpha_{m}^{}+\alpha_{1}^{-1}+\cdots+\alpha_{m}^{-1}=\beta\}|
\end{align*}
(Note that $\sigma(m,q;\beta)$ is the number of components with
those equal to $\beta$ in the vector $w_{m}$ (cf. (\ref{c1})). Then

$(1)$
%(20)%%%%%%%%%%%%%%%%%%%%%%%%%%%%%%%%%%%%%%%%%%%%%%%%%%%%%%%%%%%%%%%%%%%%%
\begin{align}\label{e1}
\sigma(m,q;\beta)=\sum\lambda(\alpha_{1}^{}+\cdots+\alpha_{m}^{})+q^{-1}
\{(q-1)^{m}+(-1)^{m+1}\},
\end{align}
%%%%%%%%%%%%%%%%%%%%%%%%%%%%%%%%%%%%%%%%%%%%%%%%%%%%%%%%%%%%%%%%%%%%%%%%%%
where the sum in (\ref{e1}) runs over all
$\alpha_{1}^{},\cdots,\alpha_{m}^{}\in\mathbb{F}_{q}^{*}$,
satisfying $\alpha_{1}^{-1}+\cdots+\alpha_{m}^{-1}=\beta$.\\

$(2)$
%(31)%%%%%%%%%%%%%%%%%%%%%%%%%%%%%%%%%%%%%%%%%%%%%%%%%%%%%%%%%%%%%%%%%%%%%
\begin{align}\label{f1}
\sigma(2,q;\beta)=
\begin{cases}
2q-3, & \text{if ~$\beta=0$},\\
K(\lambda;\beta^{-1})+q-3, & \text{if ~$\beta\neq0$}.
\end{cases}
\end{align}
%%%%%%%%%%%%%%%%%%%%%%%%%%%%%%%%%%%%%%%%%%%%%%%%%%%%%%%%%%%%%%%%%%%%%%%%%%

\end{proposition}

\begin{proof}

(1) can be proved just as Proposition \ref{E}(cf. \cite{D1},
Proposition 11). The details are left to the reader.

(2) If $m=2$, from (\ref{e1})

%(32)%%%%%%%%%%%%%%%%%%%%%%%%%%%%%%%%%%%%%%%%%%%%%%%%%%%%%%%%%%%%%%%%%%%%%
\begin{align}\label{g1}
\sigma(2,q;\beta)=\sum\lambda(\alpha_{1}+\alpha_{2})+q-2,
\end{align}
%%%%%%%%%%%%%%%%%%%%%%%%%%%%%%%%%%%%%%%%%%%%%%%%%%%%%%%%%%%%%%%%%%%%%%%%%%
where $\alpha_{1}$ and $\alpha_{2}$ run over all elements in
$\mathbb{F}_{q}^{*}$, satisfying
$\alpha_{1}^{-1}+\alpha_{2}^{-1}=\beta$.

\noindent
If $\beta=0$, then the result is clear. Assume now that
$\beta\neq0$. Then the sum in (\ref{g1}) is

\begin{align*}
\begin{split}
&\sum_{\alpha_{1}\in\mathbb{F}_{q}-\{0,\beta^{-1}\}}\lambda(\alpha_{1}^{}
+(\alpha_{1}^{-1}+\beta)^{-1})\\
&\qquad=\sum_{\alpha_{1}\in\mathbb{F}_{q}-\{0,\beta^{}\}}\lambda(\alpha_{1}^{-1}
+(\alpha_{1}^{}+\beta)^{-1}) \quad(\alpha_{1}\rightarrow\alpha_{1}^{-1})\\
&\qquad=\sum_{\alpha_{1}\in\mathbb{F}_{q}-\{0,1\}}
\lambda({\frac{\beta^{-1}}{\alpha_{1}^{2}+\alpha_{1}}})
\quad(\alpha_{1}\rightarrow\beta\alpha_{1}^{})\\
&\qquad=K(\lambda;\beta^{-1})-1~(cf. (\ref{d1})).
\end{split}
\end{align*}
\end{proof}
%%%%%%%%%%%%%%%%%%%%%%%%%%%%%%%%%%%%%%%%%%%%%%%%%%%%%%%%%%%%%%%%%%%%%%%%%%

The binary linear code $D_{m}$ is defined as\\
\begin{align*}
D_{m}=\{u\in\mathbb{F}_{2}^{N_{2}}|u\cdot w_{m}=0\},
\end{align*}\\
where the dot denotes the usual inner product in
$\mathbb{F}_q^{N_2}$.

%Remark 19%%%%%%%%%%%%%%%%%%%%%%%%%%%%%%%%%%%%%%%%%%%%%%%%%%%%%%%%%%%%%%%%
\begin{remark}\label{Q}
Clearly, the binary linear codes ${C}_{1}$ and $D_{1}$
coincide.\\
\end{remark}
%%%%%%%%%%%%%%%%%%%%%%%%%%%%%%%%%%%%%%%%%%%%%%%%%%%%%%%%%%%%%%%%%%%%%%%%%%

In view of Theorem \ref{G}, the dual $D_{m}^{\bot}$ of $D_{m}$ is
given by
%(33)%%%%%%%%%%%%%%%%%%%%%%%%%%%%%%%%%%%%%%%%%%%%%%%%%%%%%%%%%%%%%%%%%%%%%%
\begin{align}\label{h1}
D_{m}^{\bot}=\{d(a)=(\cdots,tr(a(\alpha_{1}^{}+\cdots+\alpha_{m}^{}
+\alpha_{1}^{-1}+\cdots+\alpha_{m}^{-1})),\cdots)|a\in\mathbb{F}_{q}\}.
\end{align}
%%%%%%%%%%%%%%%%%%%%%%%%%%%%%%%%%%%%%%%%%%%%%%%%%%%%%%%%%%%%%%%%%%%%%%%%%%
\indent

%Lemma 20%%%%%%%%%%%%%%%%%%%%%%%%%%%%%%%%%%%%%%%%%%%%%%%%%%%%%%%%%%%%%%
\begin{lemma}\label{R}
$(q-1)^{m}>2^{m}q^{\frac{m}{2}}$, for all $m\in \mathbb{Z}_{>0}$ and
$q=2^{r}\geq8$.
\end{lemma}

\begin{proof}
This can be shown, for example, by induction on $m$.
\end{proof}
%%%%%%%%%%%%%%%%%%%%%%%%%%%%%%%%%%%%%%%%%%%%%%%%%%%%%%%%%%%%%%%%%%%%%%%%%%

\indent
%Proposition 21%%%%%%%%%%%%%%%%%%%%%%%%%%%%%%%%%%%%%%%%%%%%%%%%%%%%%%%%%%%%%%
\begin{proposition}\label{S}
For $q=2^{r}$, with $r\geq3$, the map $\mathbb{F}_{q}\rightarrow
D_{m}^{\bot}(a\mapsto d(a))$ is an $\mathbb{F}_{2}$-linear
isomorphism.
\end{proposition}

\begin{proof}

The map is clearly $\mathbb{F}_{2}$-linear and onto. Let $a$ be in
the kernel of the map. Then
$tr(a(\alpha_{1}^{}+\cdots+\alpha_{m}^{}+\alpha_{1}^{-1}+\cdots+\alpha_{m}^{-1}))=0$,
for all $\alpha_{1},\cdots,\alpha_{m}\in\mathbb{F}_{q}^{*}$. Suppose
that $a\neq0$. Then, on the one hand,

%(34)%%%%%%%%%%%%%%%%%%%%%%%%%%%%%%%%%%%%%%%%%%%%%%%%%%%%%%%%%%%%%%%%%%%%%
\begin{align}\label{i1}
\sum_{\alpha_{1}^{},\cdots,\alpha_{m}^{}\in\mathbb{F}_{q}^{*}}(-1)^{tr(a(\alpha_{1}^{}
+\cdots+\alpha_{m}^{}+\alpha_{1}^{-1}+\cdots+\alpha_{m}^{-1}))}=(q-1)^{m}=N_{2}.
\end{align}
%%%%%%%%%%%%%%%%%%%%%%%%%%%%%%%%%%%%%%%%%%%%%%%%%%%%%%%%%%%%%%%%%%%%%%%%%%
On the other hand, (\ref{i1}) is equal to $K(\lambda;a)^{m}$, and so
from Weil's estimate (i.e. (\ref{a}) with $m=1$) we get
\begin{align*}
(q-1)^{m}\leq2^{m}q^{\frac{m}{2}}.
\end{align*}

But this is impossible for $q\geq8$, in view of Lemma \ref{R}
\end{proof}
%%%%%%%%%%%%%%%%%%%%%%%%%%%%%%%%%%%%%%%%%%%%%%%%%%%%%%%%%%%%%%%%%%%%%%%%%%

%%%%%%%%%%%%%%%%%%%%%%%%%%%%%%%%%%%%%%%%%%%%%%%%%%%%%%%%%%%%%%%%%%%%%%%%%%
%%%%%%%%%%%%%%%%%%%%%%%%%%%%%%%%%%%%%%%%%%%%%%%%%%%%%%%%%%%%%%%%%%%%%%%%
\section{Recursive formulas for $m$-multiple power moments of Kloostermann sums}
%%%%%%%%%%%%%%%%%%%%%%%%%%%%%%%%%%%%%%%%%%%%%%%%%%%%%%%%%%%%%%%%%%%%%%%%
%%%%%%%%%%%%%%%%%%%%%%%%%%%%%%%%%%%%%%%%%%%%%%%%%%%%%%%%%%%%%%%%%%%%%%%%%%
\indent
 We are now ready to derive, via Pless power moment identity,  a recursive formula
  for the $m$-multiple power moments of Kloosterman sums in terms of the frequencies
  of weights in $D_{m}$.

%Lemma 22%%%%%%%%%%%%%%%%%%%%%%%%%%%%%%%%%%%%%%%%%%%%%%%%%%%%%%%%%%%%%%%%%
\indent
\begin{lemma}\label{T}
For $a\in\mathbb{F}_{q}^{*}$, the Hamming weight $w(d(a))$ of $d(a)$
(cf. (\ref{h1})) can be expressed as follows:

%(35)%%%%%%%%%%%%%%%%%%%%%%%%%%%%%%%%%%%%%%%%%%%%%%%%%%%%%%%%%%%%%%%%%%%%%
\begin{align}\label{j1}
w(d(a))={\frac{N_{2}}{2}}-{\frac{1}{2}}K(\lambda;a)^{m},~with~
N_{2}=(q-1)^{m}.
\end{align}
%%%%%%%%%%%%%%%%%%%%%%%%%%%%%%%%%%%%%%%%%%%%%%%%%%%%%%%%%%%%%%%%%%%%%%%%%%

\end{lemma}

\begin{proof}
This can be shown exactly as the proof of Lemma \ref{K}.
\end{proof}
%%%%%%%%%%%%%%%%%%%%%%%%%%%%%%%%%%%%%%%%%%%%%%%%%%%%%%%%%%%%%%%%%%%%%%%%%%
\indent

%Corollary 23%%%%%%%%%%%%%%%%%%%%%%%%%%%%%%%%%%%%%%%%%%%%%%%%%%%%%%%%%%%%%
\begin{corollary}\label{U}
For $m=2$,
%(36)%%%%%%%%%%%%%%%%%%%%%%%%%%%%%%%%%%%%%%%%%%%%%%%%%%%%%%%%%%%%%%%%%%%%%
\begin{align}\label{r1}
w(d(a))={\frac{1}{2}}(q^{2}-3q+1-K_{2}(\lambda;a)) ~(cf. (\ref{n})).
\end{align}
%%%%%%%%%%%%%%%%%%%%%%%%%%%%%%%%%%%%%%%%%%%%%%%%%%%%%%%%%%%%%%%%%%%%%%%%%%
\end{corollary}
%%%%%%%%%%%%%%%%%%%%%%%%%%%%%%%%%%%%%%%%%%%%%%%%%%%%%%%%%%%%%%%%%%%%%%%%%%

The same argument leading to Proposition \ref{L} shows the next
proposition.

%Proposition%%%%%%%%%%%%%%%%%%%%%%%%%%%%%%%%%%%%%%%%%%%%%%%%%%%%%%%%%%%%
\begin{proposition}\label{V}

Let $\{D_{m,j}\}_{j=0}^{N_{2}}$ be the weight distribution of
$D_{m}$, where $D_{m,j}$ denotes the frequency of the codewords with
weight $j$ in $D_{m}$. Then

%(37)%%%%%%%%%%%%%%%%%%%%%%%%%%%%%%%%%%%%%%%%%%%%%%%%%%%%%%%%%%%%%%%%%%%%%
\begin{align}\label{k1}
D_{m,j}=\sum\prod_{\beta\in\mathbb{F}_{q}}{\binom{\sigma(m,q;\beta)}{\nu_{\beta}}},
\end{align}
%%%%%%%%%%%%%%%%%%%%%%%%%%%%%%%%%%%%%%%%%%%%%%%%%%%%%%%%%%%%%%%%%%%%%%%%%%
where the sum runs over all the sets of integers
$\{\nu_{\beta}\}_{\beta\in\mathbb{F}_{q}}~(0\leq\nu_{\beta}\leq\sigma(m,q;\beta))$,
satisfying

\begin{align*}
\sum_{\beta\in\mathbb{F}_{q}}\nu_{\beta}=j,~and~\sum_{\beta\in\mathbb{F}_{q}}
\nu_{\beta}\beta=0.
\end{align*}

\end{proposition}
%%%%%%%%%%%%%%%%%%%%%%%%%%%%%%%%%%%%%%%%%%%%%%%%%%%%%%%%%%%%%%%%%%%%%%%%%%
\indent

%Corollary 25%%%%%%%%%%%%%%%%%%%%%%%%%%%%%%%%%%%%%%%%%%%%%%%%%%%%%%%%%%%%%
\begin{corollary}\label{W}
Let $\{D_{2,j}\}_{j=0}^{(q-1)^2}$ be the weight distribution of
$D_2$, and let $q=2^r$, with $r\geq2$. Then, in view of Theorem
\ref{C} and (\ref{f1}), we have
%(38)%%%%%%%%%%%%%%%%%%%%%%%%%%%%%%%%%%%%%%%%%%%%%%%%%%%%%%%%%%%%%%%%%%%%%
\begin{align}\label{l1}
\begin{split}
D_{2,j}&=\sum{\binom{2q-3}{\nu_{0}}}\prod_{\beta\in\mathbb{F}_{q}^{*}}
{\binom{K(\lambda;\beta^{-1})+q-3}{\nu_{\beta}}}\\
&=\sum{\binom{2q-3}{\nu_{0}}}\prod_{\substack{|t|<2\sqrt{q}\\t\equiv
-1(4)}}\prod_{K(\lambda;\beta^{-1})=t}{\binom{t+q-3}{\nu_{\beta}}},
\end{split}
\end{align}
%%%%%%%%%%%%%%%%%%%%%%%%%%%%%%%%%%%%%%%%%%%%%%%%%%%%%%%%%%%%%%%%%%%%%%%%%%
where the sum runs over all the sets of nonnegative integers
$\{\nu_{\beta}\}_{\beta\in\mathbb{F}_{q}}$ satisfying

\begin{align*}
\sum_{\beta\in\mathbb{F}_{q}}\nu_{\beta}=j,~and~\sum_{\beta\in\mathbb{F}_{q}}
\nu_{\beta}\beta=0.
\end{align*}

From now on, we will assume that $r\geq3$, and hence every codeword
in $D_{m}^{\bot}$ can be written as $d(a)$, for a unique
$a\in\mathbb{F}_{q}$(cf. Proposition \ref{S}).

\end{corollary}
%%%%%%%%%%%%%%%%%%%%%%%%%%%%%%%%%%%%%%%%%%%%%%%%%%%%%%%%%%%%%%%%%%%%%%%%%%

 We now apply the Pless power moment identity in (\ref{u}) to $D_{m}^{\bot}$,
 in order to obtain the result in Theorem \ref{A} (1) about a recursive
 formula. Then the left hand side of that identity in (\ref{u}) is equal to

%(39)%%%%%%%%%%%%%%%%%%%%%%%%%%%%%%%%%%%%%%%%%%%%%%%%%%%%%%%%%%%%%%%%%%%%%
\begin{align}\label{n1}
\sum_{a\in\mathbb{F}_{q}^{*}}w(d(a))^h,
\end{align}
%%%%%%%%%%%%%%%%%%%%%%%%%%%%%%%%%%%%%%%%%%%%%%%%%%%%%%%%%%%%%%%%%%%%%%%%%%
with $w(d(a))$ given by (\ref{j1}). So (\ref{n1}) is seen to be
equal to

%(40)%%%%%%%%%%%%%%%%%%%%%%%%%%%%%%%%%%%%%%%%%%%%%%%%%%%%%%%%%%%%%%%%%%%%%
\begin{align}\label{o1}
\sum_{a\in\mathbb{F}_{q}^{*}}w(d(a))^h={\frac{1}{2^{h}}}\sum_{l=0}^{h}(-1)^{l}
{\binom{h}{l}}N_{2}^{h-l}MK^{ml}.
\end{align}
%%%%%%%%%%%%%%%%%%%%%%%%%%%%%%%%%%%%%%%%%%%%%%%%%%%%%%%%%%%%%%%%%%%%%%%%%%

On the other hand, noting that $dim_{\mathbb{F}_{2}}D_{m}=r$(cf.
Proposition \ref{S}) the right hand side of the Pless moment
identity(cf. (\ref{u})) becomes

%(41)%%%%%%%%%%%%%%%%%%%%%%%%%%%%%%%%%%%%%%%%%%%%%%%%%%%%%%%%%%%%%%%%%%%%%
\begin{align}\label{p1}
q\sum_{j=0}^{min\{N_{2},h\}}(-1)^{j}D_{m,j}\sum_{t=j}^{h}t!S(h,t)2^{-t}
{\binom{N_{2}-j}{N_{2}-t}}.
\end{align}
%%%%%%%%%%%%%%%%%%%%%%%%%%%%%%%%%%%%%%%%%%%%%%%%%%%%%%%%%%%%%%%%%%%%%%%%%%
Our result in (\ref{f}) follows now by equating (\ref{o1}) and
(\ref{p1}).

%Remark 26%%%%%%%%%%%%%%%%%%%%%%%%%%%%%%%%%%%%%%%%%%%%%%%%%%%%%%%%%%%%%%%
\begin{remark}\label{X}

If $m=2$, from the alternative expression of $w(d(a))$ in (\ref{r1})
we see that (\ref{n1}) can also be given as

%(42)%%%%%%%%%%%%%%%%%%%%%%%%%%%%%%%%%%%%%%%%%%%%%%%%%%%%%%%%%%%%%%%%%%%%%
\begin{align}\label{q1}
\sum_{a\in\mathbb{F}_{q}^{*}}w(d(a))^{h}={\frac{1}{2^{h}}}\sum_{l=0}^{h}
(-1)^{l}{\binom{h}{l}}(q^{2}-3q+1)^{h-l}MK_{2}^{l}.
\end{align}
%%%%%%%%%%%%%%%%%%%%%%%%%%%%%%%%%%%%%%%%%%%%%%%%%%%%%%%%%%%%%%%%%%%%%%%%%%

\end{remark}
%%%%%%%%%%%%%%%%%%%%%%%%%%%%%%%%%%%%%%%%%%%%%%%%%%%%%%%%%%%%%%%%%%%%%%%%%%

\indent

\bibliographystyle{amsplain}

\end{document}